%% file: parking.tex
\documentclass[10pt,a4paper,reqno]{amsart} 

\usepackage{amsmath}
\usepackage{bbm}
\usepackage{mathtools}
\usepackage{amssymb}
\usepackage{graphicx}
\usepackage{color}
\usepackage{latexsym}
\usepackage[utf8]{inputenc}
\usepackage[T1]{fontenc}
\usepackage{comment}
\usepackage{stmaryrd}

\newcommand{\ns}{{\mathbb N}} 
\newcommand{\zs}{{\mathbb Z}} 
\newcommand{\qs}{{\mathbb Q}}  

\newcommand{\bu}{\bar u}

\newcommand{\Sn}{{\mathfrak S}}

\newcommand{\GK}{\mathbb{K}}

\DeclareMathOperator{\DR}{DR}

\newcommand{\tG}{\tilde G}

\newcommand{\cT}{\mathcal T}

\newcommand{\p}{permutation}

\newcommand{\figeps}[3]
{\begin{figure}[ht!]
\begin{center} 
\includegraphics[width=#1cm]{#2.eps}\caption{#3}\label{fig:#2} 
\end{center}
\end{figure}}

\newtheorem{Theorem}{Theorem}
\newtheorem{Proposition}[Theorem]{Proposition}

\newtheorem{Lemma}[Theorem]{Lemma}

\theoremstyle{Definition}
\newtheorem{Definition}[Theorem]{Definition}

\newcommand{\beq}{\begin{equation}}
\newcommand{\eeq}{\end{equation}}

\newcommand{\gf}{generating function}
\newcommand{\gfs}{generating functions}
\newcommand{\fps}{formal power series}

\def\emm#1,{{\em #1}}

\graphicspath{{Figures/}}

\catcode`\@=11
\def\section{\@startsection{section}{1}%
 \z@{.7\linespacing\@plus\linespacing}{.5\linespacing}%
 {\normalfont\bfseries\scshape\centering}}

\def\subsection{\@startsection{subsection}{2}%
  \z@{.5\linespacing\@plus\linespacing}{.5\linespacing}%
  {\normalfont\bfseries\scshape}}

\def\subsubsection{\@startsection{subsubsection}{3}%
 \z@{.5\linespacing\@plus\linespacing}{-.5em}
  {\normalfont\bfseries\itshape}}
\catcode`\@=12

\def\cT{\mathcal{T}}
\def\cTn{\cT_n}

\newcommand{\spacebreak}
{\begin{displaymath} \triangleleft \; \lhd \;
\diamond \; \rhd \; \triangleright
  \end{displaymath}}

\begin{document}
\title
[Tamari lattices and parking functions]
{Tamari lattices and parking functions:\\
proof of a conjecture of F. Bergeron}

\author[M. Bousquet-M\'elou]{Mireille Bousquet-M\'elou}
\author[G. Chapuy]{Guillaume Chapuy}
\author[L.-F. Préville-Ratelle]{Louis-François Préville-Ratelle}

\address{MBM: CNRS, LaBRI, Universit\'e Bordeaux 1, 
351 cours de la Lib\'eration, 33405 Talence, France}
\email{mireille.bousquet@labri.fr}
\address{GC: CNRS, LIAFA, Universit\'e Paris Diderot - Paris 7, Case 7014,
75205 Paris Cedex 13, France}
\email{guillaume.chapuy@liafa.jussieu.fr}
\address{LFPR: LACIM, UQAM, C.P. 8888 Succ. Centre-Ville, Montréal H3C 3P8, Canada}
\email{preville-ratelle.louis-francois@courrier.uqam.ca}
%

\thanks{GC and LFPR were partially supported by the European project
  ExploreMaps -- ERC StG 208471. LFPR is also partially supported by an
Alexander Graham Bell Canada Graduate Scholarship from the Natural
  Sciences and Engineering Research Council of Canada (NSERC)}


\begin{abstract}
An {$m$-ballot path} of size $n$ is a path  on the square grid
consisting of north and east unit steps, starting at
$(0,0)$,  ending at $(mn,n)$, and never going below the line
$\{x=my\}$. The set of these paths can be equipped with a lattice structure,
called the $m$-Tamari lattice and denoted by $\cTn^{(m)}$, which
generalizes the usual Tamari  
lattice $\cTn$ obtained when $m=1$. This lattice was  
introduced by F.~Bergeron   in connection with the
study of coinvariant spaces. He conjectured several intriguing
formulas dealing with the enumeration of intervals in this
lattice. One of them states that the number of intervals in
$\cTn^{(m)}$ is
$$
\frac {m+1}{n(mn+1)} {(m+1)^2 n+m\choose n-1}.
$$
This conjecture was proved recently, but in a non-bijective way, while
its form strongly suggests a connection with plane trees.

 Here, we prove another conjecture of Bergeron, which deals with the
 number of \emm labelled, intervals. An interval $[P,Q]$ of 
$\cTn^{(m)}$ is \emm labelled, if the north steps of $Q$ are labelled
from 1 to $n$ in such a way the labels increase along any sequence of
consecutive 
north steps. We prove that the number of  labelled intervals in
$\cTn^{(m)}$ is 
$$
{(m+1)^n(mn+1)^{n-2}}.
$$
The form of these numbers suggests a connection with parking functions,
but our proof is non-bijective. It is based on a recursive description of
intervals, which translates into a functional equation satisfied by
the associated \gf. This equation involves a derivative and a divided
difference, taken with respect to two additional variables. Solving
this equation is the hardest part of the paper.

Finding a bijective proof remains an open problem.
\end{abstract}

\date{\today}
\maketitle


\section{Introduction and main results}
An \emph{$m$-ballot path} of size $n$ is a path  on the square grid
consisting of north and east unit steps, starting at
$(0,0)$,  ending at $(mn,n)$, and never going below the line
$\{x=my\}$. 
It is well-known that  there are 
$$\frac
1{mn+1}{(m+1)n \choose n}$$ such paths~\cite{dvoretzky}, and that they
are in bijection with $(m+1)$-ary trees with $n$ inner nodes.  
Bergeron recently defined on the  set $\cT_n^{(m)}$ of $m$-ballot
paths of size $n$ a partial order.
It is convenient to  describe it via the associated covering relation,
 exemplified in Figure~\ref{fig:push_mWalk}.
\begin{Definition}
\label{def-m-tamari}
Let $P$ and $Q$ be two $m$-ballot paths of size $n$.
 Then $ Q$ covers $P$ if  there exists in  $P$
 an east step $a$, followed by a north step $b$, such that $Q$ is
 obtained from $P$ by swapping $a$ and $S$, 
 where $S$ is the shortest factor of $P$ that begins with $b$ and is
 a (translated) $m$-ballot path.
\end{Definition}

\figeps{12}{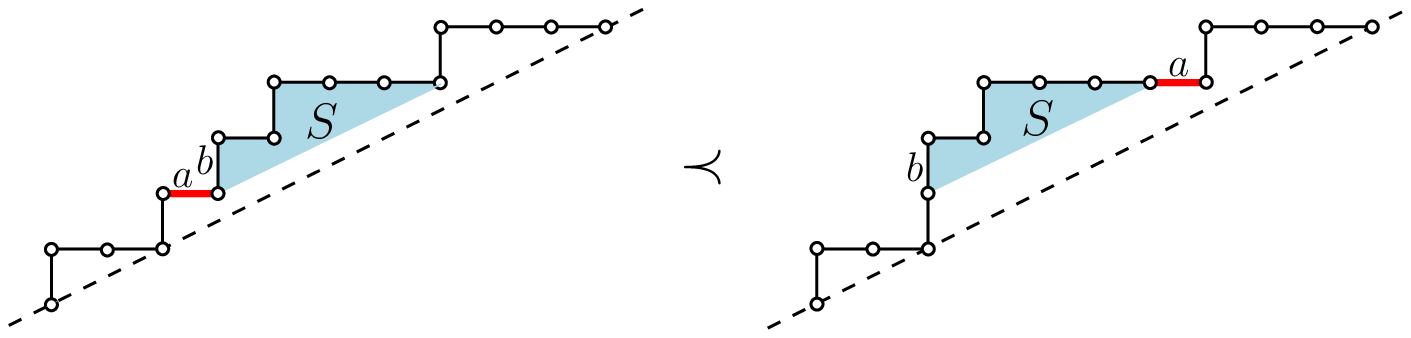}{The covering relation  between
  $m$-ballot paths ($m=2$).}

It was shown in~\cite{bousquet-fusy-preville} that
this order endows $\cT_n^{(m)}$ with a lattice
structure, which is called the \emm $m$-Tamari lattice of size
$n$,.   When $m=1$,  it coincides with the classical Tamari
lattice~\cite{BeBo07,friedman-tamari,HT72,knuth4}. Figure~\ref{fig:lattice_ex} 
shows two of the lattices $\cT_n^{(m)}$.
\begin{figure}[b]
\begin{center}
\includegraphics[height=9cm]{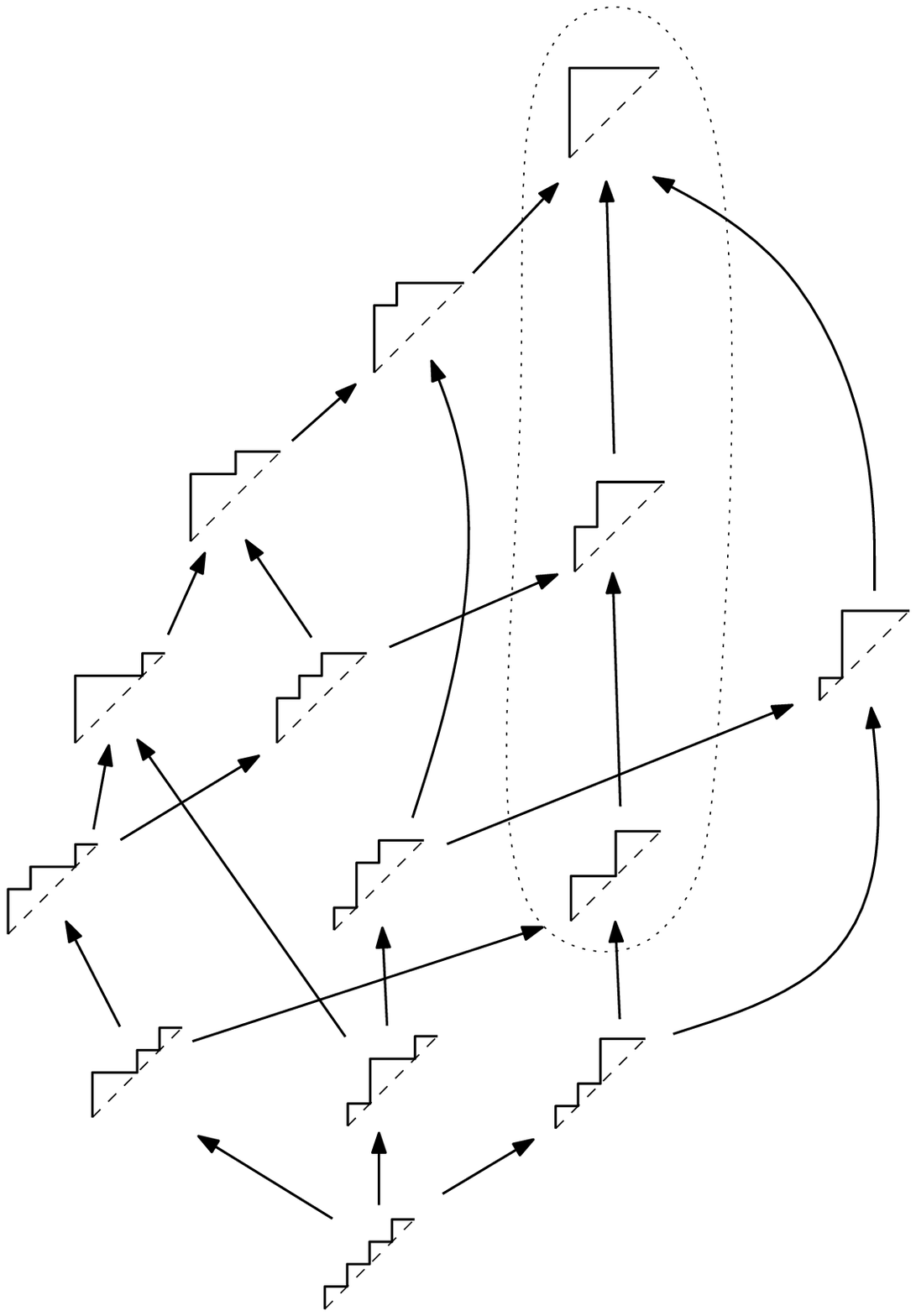}\hspace{10mm}\includegraphics[height=9cm]{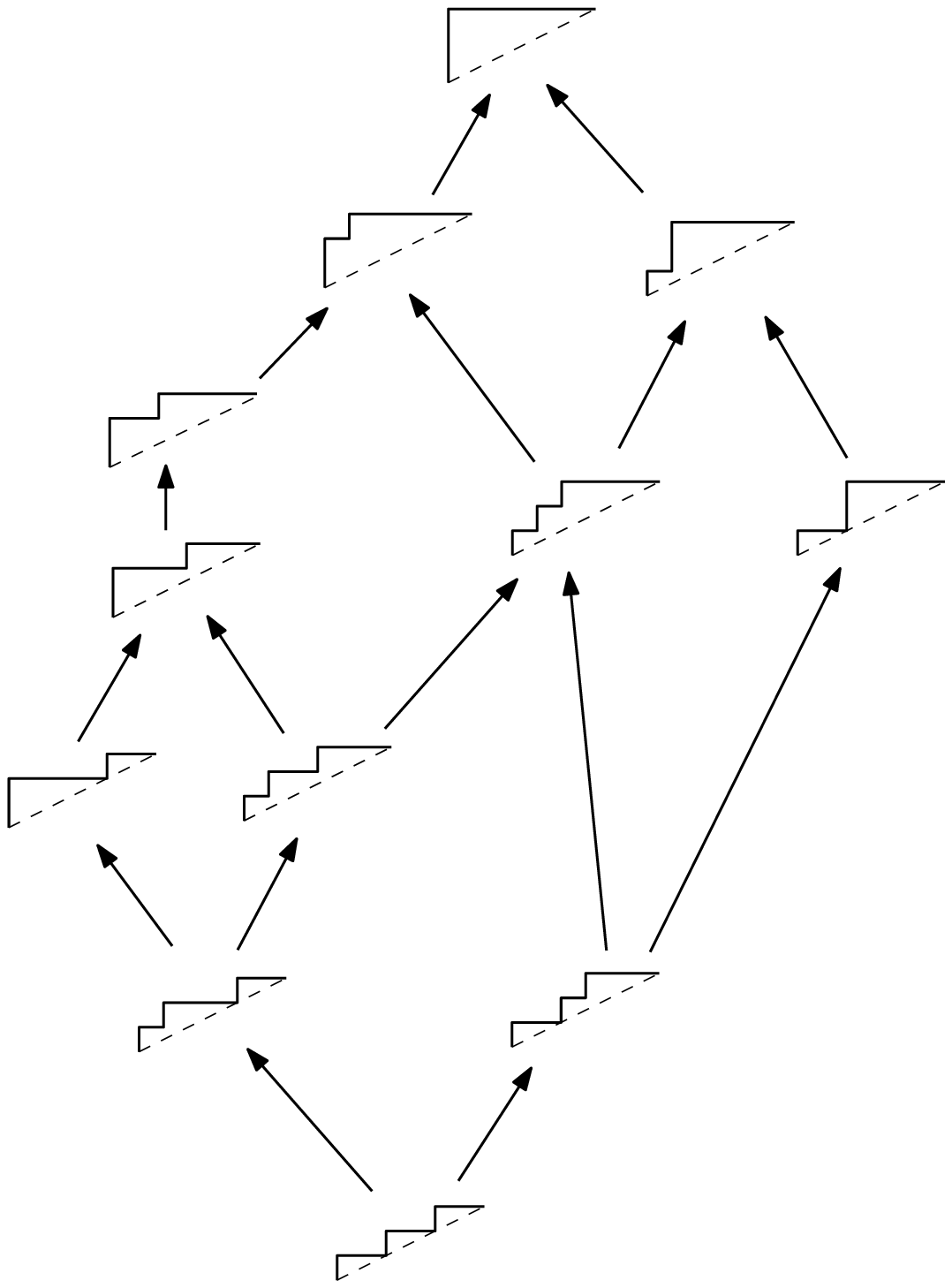}
\end{center}
\caption{The $m$-Tamari lattice $\cT_{n}^{(m)}$ for  $m=1$ and $n=4$ (left)
and for $m=2$ and $n=3$ (right). 
The three walks surrounded by a line in $\cT_{4}^{(1)}$ form a lattice
that is isomorphic to $\cT_{2}^{(2)}$ (see Proposition~\ref{prop:sublattice}).}
\label{fig:lattice_ex}
\end{figure}

These lattices
are conjectured to have deep connections with the ring $\DR_{3,n}$ of polynomials in
three sets of variables $X=\{x_1, \ldots, x_n\}$, $Y=\{y_1, \ldots,
y_n\}$, $Z=\{z_1, \ldots, z_n\}$, quotiented by the ideal generated by
(trivariate) \emm diagonal invariants,. By diagonal invariants, one means constant term free
polynomials that are invariant under the following action of the
symmetric group $\Sn_n$: for $\sigma \in \Sn_n$ and $f$ a polynomial,
$$
\sigma(f(X,Y,Z))=f(x_{\sigma(1)}, \ldots, x_{\sigma(n)},y_{\sigma(1)}, \ldots,
y_{\sigma(n)},z_{\sigma(1)}, \ldots, z_{\sigma(n)}).
$$
We refer to~\cite{bergeron-preville,bousquet-fusy-preville} for details
about these conjectures, which have striking analogies with the much
studied case of \emm two, sets of
variables~\cite{haglund-book,haglund-polynomial,HaimanPreu,HaiConj,loehr-thesis}. 
In particular, it seems that the role played by ballot paths for two
sets of variables (see, \emm
e.g.,,~\cite{MR1935784,MR1972636,MR2163448}) 
is played for three sets of variables by \emm intervals, of ballot paths in the
Tamari order.

For instance, it is conjectured in~\cite{bergeron-preville} that the
 the dimension of a certain polynomial ring related to  $\DR_{3,n}$,
 but involving one more parameter $m$, is
\beq\label{number-unlabelled}
\frac {m+1}{n(mn+1)} {(m+1)^2 n+m\choose n-1},
\eeq
and that this number counts intervals in the Tamari lattice
$\cT_n^{(m)}$ . 
The latter statement was proved
in~\cite{bousquet-fusy-preville} (the special case $m=1$ had been proved
earlier~\cite{ch06}). The former one is presumably extremely 
difficult, given the complexity of the corresponding result for two
sets of variables~\cite{HaimanPreu}. 
The dimension related result was observed earlier for small values of $n$ by
Haiman~\cite{HaiConj} in the case $m=1$. 

\figeps{4}{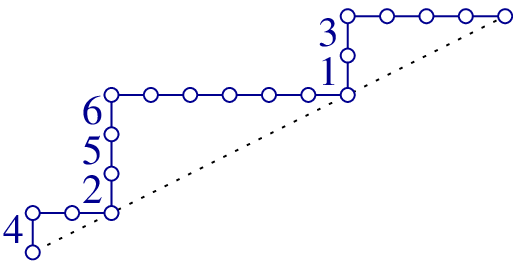}{A labelled $2$-ballot path.}

The aim of this paper is to prove another conjecture
of~\cite{bergeron-preville}, dealing with \emm labelled, Tamari
intervals. Let us say that an $m$-ballot path of size 
$n$ is \emm labelled, if the north steps are labelled from 1 to $n$,
in such a way the labels increase along any sequence of consecutive
north steps (Figure~\ref{fig:labelled}). The number of labelled
$m$-ballot paths of size $n$ is
$$
(mn+1)^{n-1}.
$$
Indeed, these paths are  in bijection with  $(1,m,\ldots,
m)$\emm -parking functions, of size $n$, in the sense 
of~\cite{pitman-stanley,yan}: the function $f$ associated with a  path $Q$ satisfies
$f(i)=k$ if the north step of $Q$ labelled $i$ lies at abscissa
$k-1$. Now, we say that an
$m$-Tamari  interval $[P,Q]$ is  labelled, if the upper path $Q$ is
labelled. It is conjectured in~\cite{bergeron-preville} that the
number of labelled $m$-Tamari intervals of size $n$ is 
\beq\label{number}
{(m+1)^n(mn+1)^{n-2}},
\eeq
and this is what we prove in this paper. It is also
conjectured in~\cite{bergeron-preville} that this number is the
dimension of a certain polynomial ring generalizing  $\DR_{3,n}$
(which corresponds to the case $m=1$). 

Our  proof is, at first blush, analogous to the proof
of~\eqref{number-unlabelled} presented in~\cite{bousquet-fusy-preville}: we introduce  a
\gf\ $F^{(m)}(t;x,y)$ counting labelled intervals according to three
parameters; we describe a
recursive construction of intervals and  translate it into a functional
equation defining $F^{(m)}(t;x,y)$ ;  we finally solve this equation,
after having partially guessed  its solution. However,  the labelled
case turns out to be significantly more 
difficult than the unlabelled one. It is not hard to explain the origin of this
increased difficulty: for $m$ fixed,
the \gf \ of the numbers~\eqref{number-unlabelled} is an \emm
algebraic, series, and can be expressed in terms of the series $Z\equiv
Z(t)$ satisfying 
$$
Z=\frac t{(1-Z)^{m(m+2)}}.
$$
There exists a wealth of tools, both modern or ancient, to handle
algebraic series (\emm e.g.,, factorisation, elimination, Gröbner bases,
rational parametrizations when the genus is zero, efficient guessing
techniques, all tools made
effective in  {\sc Maple} and its packages, like {\tt algcurves} and {\tt
  gfun}). Such tools play a key role in the proof
of~\eqref{number-unlabelled}. But the \gf\ of the
numbers~\eqref{number} is related to the series $Z$ 
satisfying
$$
Z=t e^{m(m+1)Z},
$$
which lives in the far less polished world of \emm differentially algebraic,
series, for which much fewer tools are available.

\medskip

\begin{figure}[t]
  \begin{center}
    \includegraphics{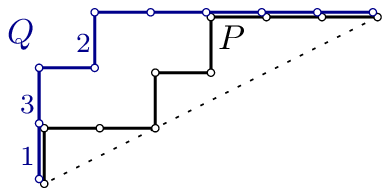}
  \end{center}
  \caption{A labelled $2$-Tamari interval $I=[P,Q]$ of size $|I|=3$. It has $c(P)=3$
  contacts, and its initial rise is $r(Q)=2$.}\label{fig:interval}
\end{figure}

Our main result is actually more general that~\eqref{number}. Indeed, we refine
the enumeration by taking into account two more parameters, which we now
define. 
A \emph{contact} of an $m$-ballot path $P$ is a vertex of $P$ lying
on the line $\{x=my\}$. 
The \emph{initial rise} of a ballot path $Q$ is the length of the initial run of up steps in $Q$.
A \emm contact, of a
Tamari interval $[P, Q]$ is a contact of the \emm lower, path~$P$,
 while the \emm initial rise, of this interval  is the initial rise of
 the \emm upper, path $Q$ (see 
 Figure~\ref{fig:interval}). 
 We consider the exponential \gf\ of
 labelled $m$-Tamari intervals, counted by the size, the number of
 contacts and the initial rise. More precisely,
\beq\label{F-def}
F^{(m)}(t;x,y)= \sum_{I=[P,Q]} \frac{t^{|I|}}{|I|!}\, x^{c(P)}y^{r(Q)},
\eeq
where the sum runs over all labelled $m$-Tamari intervals $I$, $|I|$
denotes the size of $I$ (that is, the number of up steps in $P$), $c(P)$ the
number of contacts of $P$ and $r(Q)$ the initial rise of $Q$.
The main result of this paper is a complicated closed form expression of
$F^{(m)}(t;x,y)$, which becomes  simple when $y=1$. In particular,
extracting the $n^{\hbox{th}}$ 
coefficient in   $F^{(m)}(t;1,1)$ proves Bergeron's
conjecture~\eqref{number}.
\begin{Theorem}\label{thm:main}
Let  $F^{(m)}(t;x,y)\equiv F(t;x,y)$ be the exponential generating
function of labelled $m$-Tamari intervals, defined by~\eqref{F-def}. Let $z$ and
$u$ be two indeterminates, and write 
\beq\label{t-x-param}
t=z e^{-m(m+1)z}
\quad \hbox{and } \quad x=({1+u})e^{-mzu}.
\eeq
Then $F(t;x,1)$ becomes a series in $z$ with polynomial coefficients in $u$, and this series has a simple expression:
\beq\label{Fx1}
F(t;x,1)= (1+u)e^{(m+1)z-(m-1)zu}\left(1+\frac{1-e^{mzu}}u\right).
\eeq
In particular, 
$$F(t;1,1)=(1-mz)e^{(m+1)z},$$
  and the number of labelled $m$-Tamari intervals  
of size $n$ is 
$$
n! [t^n] F(t;1,1)={(m+1)^n(mn+1)^{n-2}}.
$$
\end{Theorem}
Our expression of $F^{(m)}(t;x,y)$ is given in Theorem~\ref{thm:trivariate}. When
$m=1$, it takes a reasonably simple form, which we now present (the
case $m=2$ is also detailed at the end of the paper). Given
a Laurent polynomial $P(u)$ 
in $u$, we denote by $[u^\ge]P(u)$ the \emm non-negative part of
$P(u)$ in $u$,, defined by
$$
[u^\ge]P(u) = \sum_{i\ge 0} P_iu^i \quad \hbox{ if } \quad P(u)= \sum
_{i\in \zs} P_i u^i.
$$
The definition is then extended by linearity to power series
whose coefficients are Laurent polynomials in $u$.
\begin{Theorem}\label{thm:1}
Let  $F^{(1)}(t;x,y)\equiv F(t;x,y)$ be the generating function of
labelled $1$-Tamari intervals, defined by~\eqref{F-def}.   Let $z$ and
$u$  be two indeterminates, and set 
\beq\label{t-x-param-1}
t=z e^{-2z}
\quad \hbox{and } \quad x=({1+u})e^{-zu}.
\eeq
Then $F(t;x,y)$ becomes a formal power series in $z$ with polynomial coefficients
in $u$ and $y$,  which is given by
\beq\label{F-param-y1}
F(t;x,y)= (1+u)\,e^{2yz}\ [u^{\ge}] \Big(e^{zu(y-1)+zy\bu} - \bu e^{z\bu(y-1)+zyu} \Big),
\eeq
with $\bu=1/u$.
Equivalently,
$$
\frac{F(t;x,y)}{ (1+u)e^{2yz}}= \sum_{0\le i \le j} u^{j-i} \frac{z^{i+j}y^i (y-1)^j}{i! j!}
- 
\sum_{0\le j < i} u^{i-j-1} \frac{z^{i+j}y^i (y-1)^j}{i! j!}.
$$
\end{Theorem}
It is easily seen that the case $y=1$ of the above formula reduces to
the case $m=1$ of~\eqref{Fx1}. When $x=1$, that is, $u=0$, the double
sums in the expression of $F(t;x,y)$ reduce to simple sums, and the
\gf\ of labelled Tamari intervals is expressed in terms of Bessel functions:
$$
\frac{F(t;1,y)}{ e^{2yz}}= \sum_{i\ge 0 }  \frac{z^{2i}y^i (y-1)^i}{i!^2}
- 
\sum_{ j\ge 0 } \frac{z^{2j+1}y^{j+1} (y-1)^j}{(j+1)! j!}.
$$

\medskip
The outline of the paper goes as follows: in Section~\ref{sec:eq} we
derive from a recursive description  of labelled Tamari intervals a
functional equation satisfied by their \gf. This equation involves a
derivative (with respect to $y$) and a divided difference (with
respect to $x$). We present in Section~\ref{sec:sol1} the principle of the
proof, and exemplify it on the case $m=1$, thus obtaining
Theorem~\ref{thm:1} above. Section~\ref{sec:sol} deals with the general
case, and proves Theorem~\ref{thm:main}. 

\medskip
We conclude this introduction with some notation and a few
definitions.
Let $\GK$ be a commutative ring and $t$ an indeterminate. We denote by
$\GK[t]$ 
(resp. $\GK[[t]]$) 
the ring of polynomials 
(resp. \fps) 
in $t$ with coefficients in $\GK$. If $\GK$ is a field, then $\GK(t)$ denotes the field
of rational functions in $t$, and $\GK((t))$ the field of Laurent series
in $t$ (that is, series of the form $\sum_{n\ge n_0} a_n t^n$, with $n_0\in \zs$). These notations are generalized to polynomials, fractions
and series in several indeterminates. We 
denote by bars the reciprocals of variables: for instance, $\bu=1/u$, so that $\GK[u,\bu]$ is the ring of Laurent
polynomials in $u$ with coefficients in $\GK$.
The coefficient of $u^n$ in a Laurent  series $F(u)$ is denoted
by $[u^n]F(x)$.

We have defined the non-negative part of a Laurent polynomial $P(u)$
above Theorem~\ref{thm:1}. We define similarly the positive
part of $P(u)$, 
 denoted by $[u^>]P(u)$. 

The series we  handle in this paper involve a main variable $t$, or $z$
after the change of variables~\eqref{t-x-param}, and then  additional
variables $x$ and $y$. So they should in principle be denoted
$F(t;x,y)$, but  we often omit the variable $t$ (or $z$), to
avoid heavy notation and enhance role of the additional variables $x$
and $y$.

\section{A functional equation}
\label{sec:eq}

The aim of this section is to describe a recursive decomposition of labelled
$m$-Tamari intervals, and to translate it into a functional equation satisfied
by the associated generating function. Our description of the decomposition is
self-contained, but we refer to~\cite{bousquet-fusy-preville} 
for several proofs and details.

\subsection{Recursive decomposition of Tamari intervals}
We start by modifying the appearance of 1-ballot paths. 
We apply a 45 degree rotation to 1-ballot paths to transform them into \emm
Dyck paths,.  A Dyck path of size $n$ consists  of steps $u=(1,1)$ (up steps) and
steps $d=(1,-1)$ (down steps), starts at $(0,0)$, ends at $(2n,0)$ and never goes
below the $x$-axis.  We say that an up step has \emm rank, $i$ if it is the
$i^{\hbox{\small th}}$ up step of the path. We often represent Dyck
paths by words on the alphabet $\{u,d\}$.

Consider now an $m$-ballot path of size $n$, and  replace each north step
 by a sequence of $m$ north steps. This gives a 1-ballot path of size $mn$, and
 thus, after a rotation, a Dyck path. In this path, for each
 $i\in\llbracket 0,n-1\rrbracket$, 
the up steps of ranks $mi+1,\ldots,m(i+1)$
are consecutive. We call the Dyck paths satisfying this property
\emph{$m$-Dyck paths}, and say that the up steps
of ranks $mi+1,\ldots,m(i+1)$ form a \emph{block}. Clearly, $m$-Dyck
paths of size $mn$ (\emm i.e.,, having $n$ blocks) are in
one-to-one correspondence with $m$-ballot paths of size $n$.
We often denote by $\cT_n$, rather than $\cT_n^{(1)}$, the usual
Tamari lattice of size $n$. Similarly, the intervals of this lattice are called
Tamari intervals, rather than 1-Tamari intervals.
As proved in~\cite{bousquet-fusy-preville}, the 
transformation of $m$-ballot paths into $m$-Dyck paths maps $\cT_n^{(m)}$ on a sublattice of $\cT_{mn}$.
\begin{Proposition}[{\cite[Prop.~4]{bousquet-fusy-preville}}]
\label{prop:sublattice}
The set of $m$-Dyck paths with $n$ blocks is the sublattice of
$\mathcal{T}_{nm}$ consisting of the paths that are larger than or
equal to $u^m d^m \ldots u^m d^m$. It is
order isomorphic to $\mathcal{T}_{n}^{(m)}$. 
%
\end{Proposition}

We now describe a recursive decomposition of (unlabelled) Tamari intervals,
again borrowed from~\cite{bousquet-fusy-preville}.
Thanks to the embedding of $\cTn^{(m)}$ into
$\mathcal{T}_{nm}$, it will also enable us to decompose $m$-Tamari
intervals, for any value of $m$, in the next subsection.

A Tamari interval $I=[P,Q]$ is \emph{pointed} if its lower path $P$ has a
distinguished contact (we refer to the introduction for the definition
of contacts). Such a contact splits $P$ into two Dyck paths
$P^\ell$ and $ P^r$, respectively located to the left and to the right
of the contact.
 The pointed interval $I$ is \emph{proper} 
is $P^\ell$ is not empty,
\emm i.e.,, if the distinguished contact is not $(0,0)$.
We often use the notation $I=[P^\ell P^r, Q]$ to denote a pointed Tamari interval.
\begin{Proposition}\label{prop:decomp}
Let $I_1=[P_1^\ell P_1^r,Q_1]$  be a pointed Tamari interval, and let
$I_2=[P_2,Q_2]$ be a Tamari interval.
Construct the  Dyck paths 
$$
P=uP_1^\ell dP_1^rP_2 \quad \mbox{ and } \quad    Q=uQ_1dQ_2 
$$
as shown in Figure~{\rm\ref{fig:concatenation}}.
Then $I=[P,Q]$ is a Tamari
interval. Moreover, the mapping $(I_1,I_2)\mapsto I$ is a bijection between
pairs $(I_1,I_2)$ formed of a pointed Tamari interval 
 and a Tamari
interval,
and Tamari intervals $I$ of positive size.
Note that $I_1$ is proper if and only if the initial rise of $P$ is
not $1$.
\end{Proposition}
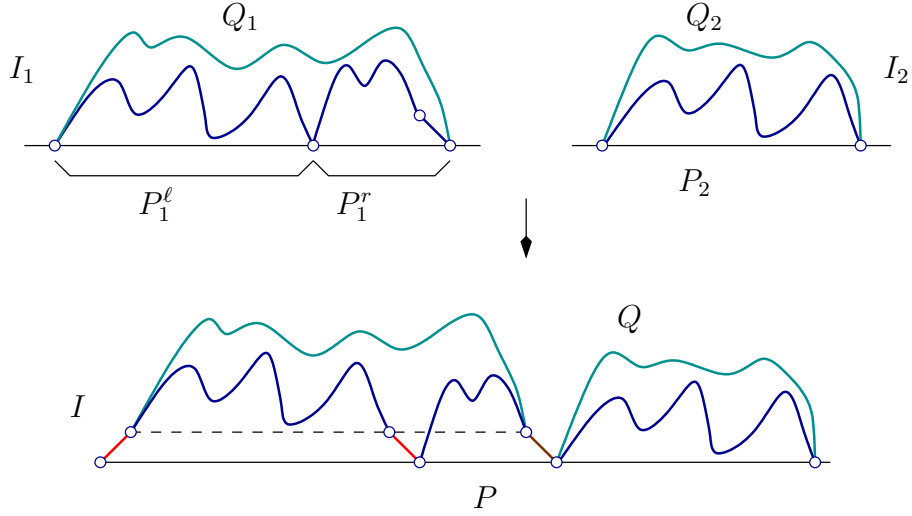
\begin{figure}
\begin{center}
\input{concatenation.pstex_t}
\end{center}
\caption{The recursive construction of Tamari intervals.}
\label{fig:concatenation}
\end{figure}
\noindent{\bf Remarks}\\
\noindent 1. To recover $P_1^\ell$, $ P_1^r$, $Q_1$, $P_2$ and $Q_2$
from $P$ and $Q$, one proceeds as follows: $Q_2$ is the part of $Q$
that follows the first return of $Q$ to the $x$-axis; this defines
$Q_1$ unambiguously. The path $P_2$ is the
suffix of $P$ having the same size as $Q_2$.  This defines $P_1:=uP_1^\ell dP_1^r$
unambiguously. Finally $P_1^r$ is the part of $P_1$
that follows the first return of $P_1$ to the $x$-axis, and this
defines $P_1^\ell$ unambiguously.
\\
\noindent 2. This proposition is obtained  by combining Proposition~5
in~\cite{bousquet-fusy-preville} and
the case $m=1$ of Lemma~9 in~\cite{bousquet-fusy-preville}. With the
notation $(P';p_1)$ and $(Q',q_1)$ used 
therein, the above paths $P_2$
and $Q_2$ are respectively the parts of $P'$ and 
$Q'$ that lie to the right of $q_1$, while $P_1^\ell  P_1^r$ and
$Q_1$ are the parts of $P'$ and 
$Q'$ that lie to the left of $q_1$. The pointed vertex $p_1$  is the
endpoint of $P_1^\ell$. Proposition~5
in~\cite{bousquet-fusy-preville} guarantees that, if $P\preceq Q$ in
the Tamari order, then $P_1^\ell  P_1^r \preceq Q_1$ and $P_2 \preceq Q_2$.
\\
\noindent 3. One can keep track of several parameters in the
construction of Proposition~\ref{prop:decomp}. For instance,
the initial rise of $Q$ equals  the initial 
rise of $Q_1$ plus one. 
 Also, the number of contacts of $P$ is 
\beq\label{eq:contacts}
c(P)=c(P_1^r )+c(P_2).
\eeq

\subsection{From the decomposition to a functional equation}
We will now establish the following functional equation.
\begin{Proposition}\label{prop:eq}
For $m\ge 1$, let  $ F^{(m)}(t;x,y)\equiv F(x,y)$ be the exponential \gf\ of
labelled $m$-Tamari intervals, defined by~\eqref{F-def}.
Then $F(x,0)=x$ and
\beq\label{eq:Fb}
\frac{\partial F}{\partial y} (x,y)=tx \left(F(x,1)\cdot  \Delta\right)^{(m)}(F(x,y)),
\eeq
where $\Delta$ is the following divided difference operator
$$
\Delta S(x)=\frac{S(x)-S(1)}{x-1},
$$
and the power $m$ means  that the operator $G(x,y)\mapsto F(x,1)\cdot \Delta
G(x,y)$ is applied $m$ times.
\end{Proposition}
\begin{proof}
We constantly use  the inclusion $\mathcal{T}_n^{(m)}\subset
\mathcal{T}_{nm}$ given by Proposition~\ref{prop:sublattice}. That is, we
identify elements of $\mathcal{T}_n^{(m)}$ with $m$-Dyck paths having
$n$ blocks.

It is obvious that $F(x,0)=x$, since the interval of size $0$ is the
only interval of initial rise~$0$, and has one contact.
The functional equation~\eqref{eq:Fb} relies on the decomposition of
Tamari intervals described in Proposition~\ref{prop:decomp}. 

We will actually apply this
decomposition to a slight
 generalization of
$m$-Tamari intervals.
For $k\geq 0$, a \emph{$k$-augmented $m$-Dyck path} is a Dyck path of 
size
 $k+mn$ for some integer $n$, where the first $k$ steps are up steps, and
 all the other up steps can be partitioned into \emph{blocks} of $m$
 consecutive up steps. 
The  $k$ first steps are not considered to be part of a block, even if $k$ is a multiple of $m$.
We denote by $\mathcal{T}^{(m,k)}$  the set of $k$-augmented $m$-Dyck paths.

A Tamari interval $I=[P,Q]$ is a \emph{$k$-augmented $m$-Tamari
  interval} if both $P$ and $Q$ belong to
$\mathcal{T}^{(m,k)}$. Assume that $P$ and $Q$ contain $n$ blocks. Then $I$
is \emm labelled, if the blocks of $Q$  are labelled from 1 to $n$
in such a way the labels increase along any sequence of consecutive
blocks. Note that labelled $0$-augmented $m$-Tamari
  intervals coincide with  labelled $m$-Tamari intervals.
Generalizing~\eqref{F-def}, we 
denote by  $F_k(t;x,y)\equiv F_k(x,y)$
the exponential generating function of labelled $k$-augmented $m$-Tamari
  intervals,
counted by the number of blocks (variable~$t$), the number of \emm
non-initial, contacts (that is, contacts distinct from $(0,0)$ ---
variable $x$), and the number of blocks contained in the first ascent
(variable $y$). 

In what follows, we  first obtain an expression of
$F_k(x,y)$ in terms of $F(x,y)$:
\beq\label{Fk-F}
 F_k(x,y) = \left\{
\begin{array}{lll}
 F(x,y)/x & \hbox{if } k=0,
\\
(F(x,1)\cdot\Delta)^{(k)}F(x,y) & \hbox{otherwise.}
\end{array}
\right.
\eeq
We then relate $m$-augmented $m$-Tamari intervals to $m$-Tamari intervals,
proving that 
\beq\label{Fm-F}
tx F_m(x,y)= \frac{\partial F}{\partial y} (x,y).
\eeq
This identity, combined with the case $k=m$ of~\eqref{Fk-F},
gives~\eqref{eq:Fb}.

\medskip
 We thus need to prove~\eqref{Fk-F} and~\eqref{Fm-F}. The case $k=0$ of~\eqref{Fk-F} is clear, since $0$-augmented
 $m$-Tamari   intervals are just  $m$-Tamari intervals. The factor $x$
 arises from the fact that $F_0(x,y)$ only keeps track of non-initial 
  contacts, while $F(x,y)$ counts all of them.

Let us now address the case $k\ge 1$ of~\eqref{Fk-F}. 
Let $I= [P,Q]$ be a labelled  
$k$-augmented $m$-Tamari interval. By Proposition~\ref{prop:decomp},
one can decompose $I$ into  a pair
$(I_1,I_2)$ of Tamari intervals, with $I_1=[P_1^\ell P_1^r,Q_1]$ and
$I_2=[P_2,Q_2]$ (see Figure~\ref{fig:concatenation}). Since the up
steps of $P_2$ and $Q_2$ 
are not in the first ascent of $P$ and $Q$, the paths $P_2$ and $Q_2$
are actually $m$-Dyck paths, so that $I_2$ is an $m$-Tamari interval.
Similarly, $I_1$ is a pointed $(k-1)$-augmented $m$-Tamari interval, which
is proper  if $k> 1$ (Proposition~\ref{prop:decomp}).

The blocks of $Q_1$ and $Q_2$ inherit a labelling from
$Q$. We  normalise these labellings in the usual way: if $P_1$ has
$n_1$ blocks and $Q_2$ has $n_2$ blocks, 
we relabel the blocks of $Q_1$ (resp.~$Q_2$) with $1,
\ldots, n_1$ (resp. $1, \ldots, n_2$) while preserving the relative
order of the labels occurring in $Q_1$  (resp.~$Q_2$).

Conversely, consider a pair $(I_1,I_2)$, where
$I_1=[P_1^\ell P_1^r,Q_1]$ is a labelled pointed  $(k-1)$-augmented $m$-Tamari
interval and $I_2$ is a labelled $m$-Tamari interval. If $k>1$, assume
moreover that $I_1$ is proper. If $Q_1$ and $Q_2$ have  respectively
$n_1$ and $n_2$  blocks,  one can reconstruct from $(I_1, I_2)$
exactly $\frac{(n_1+n_2)!}{n_1!n_2!}$ different labelled $k$-augmented
$m$-Tamari intervals $I=[P,Q]$, having $n_1+n_2$
blocks. By~\eqref{eq:contacts}, the number of non-initial 
contacts in $I$ is the number of non-initial contacts in $P_1^r$, plus
the number of contacts in $P_2$. The number of blocks in the first
ascent of $Q$ is the number of blocks in the first ascent of $Q_1$.

The exponential \gf\ of labelled $m$-Tamari intervals $I_2$, counted by the
size and the number of contacts, is $F(x,1)$. Let
$F_{k-1}^\circ(t;x,y)\equiv F_{k-1}^\circ(x,y)$ be the exponential
\gf\  of labelled \emm proper, pointed  $(k-1)$-augmented $m$-Tamari 
intervals $I_1=[P_1^\ell P_1^r,Q_1]$, counted by the size, the number
of non-initial contacts in $P_1^r$ and the number of blocks in the
first ascent of $Q_1$. Note that the \gf\ of labelled \emm non-proper,
pointed $0$-augmented $m$-Tamari 
intervals is simply $F_0(x,y)$. The above construction then implies that
\beq\label{Fk-Fcirc}
F_k(x,y)= F(x,1) \left( F_{k-1}^\circ(x,y)+ F_0(x,y){\mathbbm 1}_{k=1}\right).
\eeq
We claim that  
\beq\label{Fcirc}
F_{k-1}^\circ(x,y)=
\Delta F_{k-1}(x,y).
\eeq
It follows from~(\ref{Fk-Fcirc}--\ref{Fcirc}) that
$$
F_k(x,y)= \left\{
  \begin{array}{lll}
   F(x,1) \cdot \Delta \big(xF_0(x,y)\big)& \mbox{if } k=1,\\
  F(x,1) \cdot \Delta F_{k-1}(x,y) &\mbox{otherwise}.
  \end{array}
 \right.
$$
Now  a simple induction on $k>0$, combined with the case $k=0$ of~\eqref{Fk-F} and, proves the case $k > 0$
of~\eqref{Fk-F}. So~\eqref{Fk-F} will be proved if we establish~\eqref{Fcirc}. 
Write
$$
F_{k-1}(x,y)= \sum_{i\ge 0} F_{k-1,i}(y) x^i,
$$
so that $F_{k-1,i}(y)$ counts labelled $(k-1)$-augmented $m$-Tamari intervals
having $i$ non-initial contacts. By pointing a non-initial contact, such an interval gives rise to $i$
labelled proper pointed  $(k-1)$-augmented $m$-Tamari 
intervals $[P_1^\ell P_1^r,Q_1]$, having respectively $0, 1, \ldots,
i-1$ non-initial contacts in $P_1^r$. Hence
\begin{eqnarray*}
 F_{k-1}^\circ(x,y)&=& \sum_{i\ge 0} F_{k-1,i}(y) \left(1+x+ \cdots
+x^{i-1}\right),
\\
&= &
\sum_{i\ge 0} F_{k-1,i}(y) \frac{x^i -1}{x-1}
\\
&=& \frac{F_{k-1}(x,y)-F_{k-1}(1,y)}{x-1} \\
&=& \Delta F_{k-1}(x,y).
\end{eqnarray*}
This coincides with~\eqref{Fcirc}.

\medskip
We finally want to prove~\eqref{Fm-F}, and this will complete the
proof of Proposition~\ref{prop:eq}. A labelled $m$-augmented
$m$-Tamari interval $I=[P,Q]$ having $n-1$ blocks gets in $F_m(x,y)$ a weight
$$
\frac{t^{n-1}}{(n-1)!}\,x^{c(P)-1} y^{r(Q)-1}
= \frac n{tx}\,\left( \frac{t^n}{n!}\,x^{c(P)} y^{r(Q)-1}\right),
$$
where $r(Q)$ is the initial rise of $Q$, divided by $m$. 
Let us interpret the factor $n$ as the choice of a label
$i\in\llbracket 1, n\rrbracket$ assigned
to the  first $m$ steps of $P$, while the labels of the
blocks, which were $1, \ldots, n-1$, are redistributed so as to avoid
$i$. 
The above identity shows that  $tx F_m(x,y)$ counts (by the number of blocks, the number
of contacts, and the initial rise minus one), $m$-Tamari intervals
$[P,Q]$ in which the blocks are labelled in such a way the labels
increase along sequences of consecutive blocks, except that the first
block of the first ascent may have a larger label than the second
block of the first ascent. Such intervals are
obtained from  usual labelled $m$-Tamari intervals by choosing a block
in the first ascent and exchanging its label with the label of the
very first block.
In terms of power series, choosing a block of the first ascent boils down to
differentiating with respect to $y$ (this also decreases by 1 the
exponent of $y$), and we thus obtain~\eqref{Fm-F}.
\end{proof}

\section{Principle of the proof, and  the case $m=1$}
\label{sec:sol1}
\subsection{Principle  of the proof}
\label{sec:principle}
Let us consider the functional equation~\eqref{eq:Fb}, together with
the initial condition  $F(t;x,0)=x$. Perform the change of
variables~\eqref{t-x-param}, and denote 
$G(z;u,y)\equiv G(u,y)= F(t;x,y)$. Then $G(u,y)$ is a series in $z$
with coefficients in $\qs[u,y]$, satisfying 
\beq\label{eq:G}
\frac{\partial G}{\partial y} (u,y)= 
z(1+u) e^{-mzu-m(m+1)z}\left( \frac {uG(u,1)}{(1+u)e^{-mzu}-1}
\ \Delta_u\right)^{(m)} G(u,y),  
\eeq
with
$
\Delta_u H(u)= \frac {H(u)-H(0)}{u},
$
and the initial condition 
\beq\label{init-G}
G(u,0)=(1+u)e^{-mzu}.
\eeq
Observe that this pair of equations defines $G(z;u,y)\equiv G(u,y)$ uniquely as a
\fps\ in $z$. Indeed, the coefficient of $z^n$ in $G$
can be computed inductively from these equations (one first determines
the coefficient of $z^n$ in $\frac{\partial G}{\partial y}$,
which can be expressed, thanks to~\eqref{eq:G}, in terms of the
coefficients of $z^i$ in  $G$ for $i<n$. Then the coefficient of $z
^n$ in $G$ is obtained by integration with respect to $y$, using the
initial condition~\eqref{init-G}).
Hence, if we exhibit
a series $\tilde G(z;u,y)$ that satisfies both equations, then
 $\tilde G(z;u,y)=G(z;u,y)$. We are going to construct such a series.

Let 
\beq\label{tG1}
G_1(z;u)\equiv G_1(u)= (1+u)e^{(m+1)z-(m-1)zu}\left(1+\frac{1-e^{mzu}}u\right).
\eeq
Then $G_1(u)$ is a series in $z$ with polynomial coefficients in $u$,
which, as we will see, coincides with $G(u,1)$. Consider now the
following  equation, obtained from~\eqref{eq:G} by replacing
$G(u,1)$ by its conjectured value $G_1(u)$:
\beq\label{eq:Gtilde0}
\frac{\partial \tilde G}{\partial y} (z;u,y)= 
z(1+u) e^{-mzu-m(m+1)z}\left( \frac
{uG_1(u)}{(1+u)e^{-mzu}-1}\ \Delta_u\right)^{(m)} \tG(z;u,y), 
\eeq
with the initial condition 
\beq\label{init:Gtilde}
\tilde G(z;u,0)=(1+u)e^{-mzu}.
\eeq 
 Eq.~\eqref{eq:Gtilde0}  can be rewritten as 
\beq\label{eq:Gtilde}
\frac{\partial \tilde G}{\partial y} (z;u,y)= z(1+u)e^{-mzu} \Lambda^{(m)}\tilde G(z;u,y)
\eeq
where 
$\Lambda$ is the  operator defined by
\beq\label{Lambda-def}
\Lambda (H) (u) = 
\frac{H(u)-H(0)}{A(u)} 
\eeq
with \beq\label{A-def}
A(u)=\frac{u}{1+u}e^{-zu},
 \eeq
and $\Lambda^{(m)}$ denotes the $m$-th iterate of  $\Lambda$.
Again, it is not hard
to see that~\eqref{eq:Gtilde} and the initial
condition~\eqref{init:Gtilde} define a unique 
series in $z$, denoted $\tG(z;u,y)\equiv \tG(u,y)$. The coefficients
of this series lie 
in $\qs[u,y]$. The principle of our proof can be described as follows.

\begin{quote}
{\it  If we prove that $\tG(u,1)=G_1(u)$, then the
  equation~{\rm\eqref{eq:Gtilde0}} satisfied by $\tG$ coincides with the
  equation~\eqref{eq:G} that defines $G$, and thus
  $\tG(u,y)=G(u,y)$. In particular,
  $G_1(z;u)=\tG(z;u,1)=G(z;u,1)=F(t;x,1)$, and 
  Theorem~{\rm\ref{thm:main}} is proved. }
\end{quote}

\subsection{The case $m=1$}
\label{subsec:sol1}
Take $m=1$. In this subsection, we describe the three steps that, starting
from~\eqref{eq:Gtilde},  prove that
$\tG(u,1)=G_1(u)$. In passing, we establish the expression~\eqref{F-param-y1} of
$F(t;x,1)$ (equivalently, of $\tG(z;u,1)$) given in  
Theorem~\ref{thm:1}.

\subsubsection{A homogeneous differential equation and its solution}
\label{sec:m11}
When $m=1$, the equation~\eqref{eq:Gtilde}  defining $\tG(z;u,y)\equiv
\tG(u,y)$ reads 
\beq\label{ED-uy}
\frac{\partial \tG}{\partial y} (u,y)= z(1+u)(1+\bu) ({\tG(u,y)-\tG(0,y)}),
\eeq
where $\bu=1/u$, with the initial condition 
\beq\label{init-1}
\tG(u,0)=({1+u})e^{-zu}.
\eeq
These equations imply  that $\tG(-1,y)=0$.  The
coefficient of $\tG(u,y)$ in the right-hand side of~\eqref{ED-uy} is symmetric in $u$ and $\bu$. We are going
to exploit this symmetry to eliminate the term $\tG(0,y)$. 
Replacing $u$ by $\bu$ in~\eqref{ED-uy} gives
$$
\frac{\partial \tG}{\partial y} (\bu,y)= z(1+u)(1+\bu)
({\tG(\bu,y)-\tG(0,y)}),
$$
so that
$$
\frac{\partial}{\partial y} \left( \tG(u,y)-\tG(\bu,y)\right)= z(1+u)(1+\bu)
\left(\tG(u,y)-\tG(\bu,y)\right).
$$
This is a homogeneous linear differential equation satisfied by
$\tG(u,y)-\tG(\bu,y)$. It is readily solved, and  the initial
condition~\eqref{init-1}  yields
\beq\label{tGsol1}
\tG(u,y)-\tG(\bu,y)
=(1+u)e^{yz(1+u)(1+\bu)}\left(e^{-zu}-\bu e^{-z\bu}\right).
\eeq
\subsubsection{Reconstruction of $\tG(u,y)$}
Recall that $\tG(u,y)\equiv \tG(z;u,y)$ is a series in $z$ with
polynomial coefficients in $u$ and $y$. Hence, by extracting from the
above equation the positive part in $u$ (as defined at the end of
Section~1), we obtain  
$$
\tG(u,y)-\tG(0,y)=[u^>]\left((1+u)e^{yz(1+u)(1+\bu)}\left(e^{-zu}-\bu
e^{-z\bu}\right)\right).
$$
For any Laurent polynomial $P$, we have
\beq\label{1+u}
[u^>](1+u)P(u)= (1+u)[u^>]P(u)+u [u^0]P(u).
\eeq
Hence
\begin{multline*}
  \tG(u,y)-\tG(0,y)
=(1+u)[u^>]\left(e^{yz(1+u)(1+\bu)}\left(e^{-zu}-\bu
e^{-z\bu}\right)\right)
\\
+u[u^0]\left(e^{yz(1+u)(1+\bu)}\left(e^{-zu}-\bu
e^{-z\bu}\right)\right).
\end{multline*}
Setting $u=-1$ in this equation gives, since $\tG(-1,y)=0$, 
$$
-\tG(0,y)
=-[u^0]\left(e^{yz(1+u)(1+\bu)}\left(e^{-zu}-\bu e^{-z\bu}\right)\right),
$$
so that finally,
\begin{eqnarray}
\tG(u,y) 
&=&(1+u)[u^>]\left(e^{yz(1+u)(1+\bu)}\left(e^{-zu}-\bu 
e^{-z\bu}\right)\right) \nonumber
\\
&& \ \hskip 50mm +(1+u)[u^0]\left(e^{yz(1+u)(1+\bu)}\left(e^{-zu}-\bu
e^{-z\bu}\right) \right) \nonumber
\\
&=&(1+u)[u^\ge]\left(e^{yz(1+u)(1+\bu)}\left(e^{-zu}-\bu
e^{-z\bu}\right)\right). \label{Gsol-1}
\end{eqnarray}
As explained in
Section~\ref{sec:principle}, $\tG(u,y)=G(u,y)$ will be proved if we
establish that 
$\tG(u,1)=G_1(u)$. This is the final step of our proof.
\subsubsection{The case $y=1$}
\label{sec:m=y=1}
 Equation~\eqref{Gsol-1} completely describes the solution of~\eqref{ED-uy}. It remains to check that 
$\tG(u,1)=G_1(u)$, that is
\beq\label{G11conj}
\tG(u,1)= (1+u)e^{2z}\left(1+\frac{1-e^{zu}}u\right).
\eeq
Let us set $y=1$ in~\eqref{Gsol-1}. We find
\begin{eqnarray*}
  \tG(u,1)&=&(1+u)[u^\ge]\left(e^{z(2+\bu)}-\bu e^{z(2+u)}\right)\\
&=& (1+u) e^{2z}\left( 1- \frac{e^{zu}-1}u\right),
\end{eqnarray*}
which coincides with~\eqref{G11conj}. Hence
$\tG(z;u,y)=G(z;u,y)=F(t;x,y)$ (with the change 
of variables~\eqref{t-x-param}), and Theorem~\ref{thm:1} is proved.

\subsubsection{The trivariate series}
 We have now proved that $\tG(u,y)=G(u,y)$, so that
$F(x,y)=\tG(u,y)$ after the change of variables~\eqref{t-x-param-1}.  The expression~\eqref{F-param-y1} of 
$F(x,y)$ given in Theorem~\ref{thm:1} now follows from~\eqref{Gsol-1}.

\section{Solution of the functional equation: the general case}
\label{sec:sol}
We now adapt to the general case the solution described for $m=1$ in
Section~\ref{subsec:sol1}.

\subsection{A homogeneous differential equation and its solution}

Let us return to the equation \eqref{eq:Gtilde} satisfied by $\tilde G(u,y)$.
 The
coefficient of $\tG(u,y)$ in the right-hand side of this equation
is  $zv(u)$, where
$$
v(u)= (1+u)e^{-mzu}A(u)^{-m}=(1+u)^{m+1} \bu^m .$$
 In the case $m=1$, this (Laurent) polynomial was
$(1+u)(1+\bu)$, 
and took the same value for $u$ and $\bu$. 
We are again interested in the series $u_i$ such that $v(u_i)=v(u)$.
\begin{Lemma}\label{lem:ui}
  Denote $v\equiv v(u)=(1+u)^{m+1} u^{-m}$, and  consider the
  following polynomial equation in $U$:
 \beq\label{poluv-bis}
(1+U)^{m+1}=U^m v .
\eeq
This equation has no double root. We denote its $m+1$ roots by $ u_0=u,
u_1, \dots, u_m$.
\end{Lemma}
\begin{proof}
A double root of~\eqref{poluv-bis} would also satisfy
$$
(m+1) (1+u)^m =mu^{m-1}v,
$$
and this is easily shown to be impossible.
\end{proof}

\noindent{\bf Remark.} One can of course express the $u_i$'s as
Puiseux series in $u$ (see~\cite[Ch.~6]{stanley-vol2}), but this will
not be needed here, and we will think of them as abstract elements of an
 algebraic extension of $\qs(u)$. In fact, all the series in $z$ that
involve the $u_i$'s in this paper have coefficients that are \emm symmetric rational
functions, of the $u_i$'s, and hence, rational functions of $v$. At
some point, we will have to prove that a symmetric polynomial in the
$u_i$'s (and thus a polynomial in $v$) vanishes at $v=0$, that is, at
$u=-1$, and we will  consider series expansions of the $u_i$'s around $u=-1$.

\begin{Proposition}\label{prop:combi-lin-m}
 Denote $v=(1+u)^{m+1}
u^{-m}$, and let the series $u_i$ be defined as above. Denote
$A_i=A(u_i)$, where $A(u)$ is given by~\eqref{A-def}. Then 
\begin{eqnarray}\label{eq:combi-lin-y}
\sum_{i=0}^m \frac{\tilde G(u_i,y)}{\prod_{j\neq i} (A_i-A_j)} = v e^{zvy}.
\end{eqnarray}
\end{Proposition}
\noindent By $\prod_{j\neq i} (A_i-A_j)$ we mean $\prod_{0\le j \le m, j\neq i}
(A_i-A_j)$ but we prefer the shorter notation when the bounds
on $j$ are clear.
 Observe that the $A_i$'s are distinct since the $u_i$'s are distinct
(the coefficient of $z^0$ in $A(u)$ is $1/(1+\bu)$). Note also that
when $m=1$, then $u_0=u$, $u_1= \bu$, and~\eqref{eq:combi-lin-y}
coincides with~\eqref{tGsol1}. 
In order to prove the proposition, we need the following  two lemmas.
\begin{Lemma}\label{lemma:Lagrange}
Let $x_0, x_1, \dots , x_m$ be $m+1$  variables. Then
\beq\label{eq:Lagrange-inv}
\sum_{i=0}^m \frac{{x_i}^m}{\prod_{j\neq i}{(x_i-x_j)}} = 1
\eeq
and
\beq\label{eq:Lagrange-inv-bis}
\sum_{i=0}^m \frac{1/x_i}{\prod_{j\neq i}{(x_i-x_j)}} = (-1)^{m} \prod_{i=0}^m
\frac{1}{x_i}.
\eeq
Moreover, for any polynomial $Q$ of degree less than $m$,
\begin{eqnarray}\label{eq:Lagrange-poly}
\sum_{i=0}^m \frac{Q(x_i)}{\prod_{j\neq i}{(x_i-x_j)}} = 0.
\end{eqnarray}
\end{Lemma}
\begin{proof}
By the Lagrange interpolation, any polynomial $R$ of degree at most $m$
satisfies:
$$
R(X)=\displaystyle\sum_{i=0}^m {R(x_i)}
\prod_{j\neq i}\frac{X-x_j}{x_i-x_j}.
$$
Equations \eqref{eq:Lagrange-inv-bis} and~\eqref{eq:Lagrange-poly} follow  by
evaluating this equation at $X=0$, respectively with $R(X)=1$
and $R(X)=X Q(X)$.
Equation~\eqref{eq:Lagrange-inv} is obtained by taking $R(X)=X^m$ and
extracting the leading coefficient. 
\end{proof}

\begin{Lemma}\label{lemma:implicit-functions} 
Let $H(u)\equiv H(z;u)\in \GK(u) [[ z]]$ 
be a formal power series in $z$ whose coefficients are rational functions in
$u$ over some field $\GK$ of
characteristic $0$. Assume that these coefficients have no pole at $u=0$.
Then there exists a sequence  $g_0(z), g_1(z), \dots $ of formal power series 
in $z$ such that for every $k\geq 0$ one has:
\beq\label{eq:Lambdaiter-A}
\Lambda^{(k)} H(u)= \frac{1}{A(u)^k} 
\left(H(u) - \sum_{j=0}^{k-1} g_j(z)A(u)^{j}\right).
\eeq
where $\Lambda$ is the operator defined
by~\eqref{Lambda-def}. 
\end{Lemma}
\begin{proof}
We denote by $\mathcal{L}$ the subring of $\GK(u)[[z]] $
  formed by formal power series whose coefficients have no pole at $u=0$. 
  By assumption, $H(u)\in \mathcal{L}$.
  We  use the notation $O(u^k)$   to denote an element of $\GK(u)[ [z]]$ of the
  form $u^k J(z;u)$ with $J(z;u)\in \mathcal{L}$.

  First, note that  $A(u)=ue^{-zu}/(1+u)$
belongs to $\GK(u)[[ z ]] $. Moreover, 
  \beq\label{eq:Adev}
   A(u)=u+ O(u^2).
   \eeq  

  We will first prove 
that there exists a sequence of formal power series
  $(g_j)_{j\geq0}\in\GK[ [ z] ]^\ns$ such that for all $\ell\geq 0$, 
  \begin{eqnarray}\label{eq:implicit-functions}
    H(u) = \sum_{j=0}^{\ell-1} g_j(z) A(u)^j + O(u^{\ell}).
  \end{eqnarray}
We will then prove that these series $g_j$ satisfy~\eqref{eq:Lambdaiter-A}.
In order to prove~\eqref{eq:implicit-functions}, we proceed by
induction on $\ell \ge 0$.   The identity holds for $\ell=0$ since
$H(u)\in\mathcal{L}$. Assume it holds for some $\ell\geq
  0$, \emm i.e.,, that there exists 
series $g_0, \ldots, g_{\ell-1}$ in $\GK[[z]]$
  and $J(u)\in\mathcal{L}$ such that
  \begin{eqnarray*}
    H(u) = \sum_{j=0}^{\ell-1} g_j(z) A(u)^j + u^{\ell} J(u).
  \end{eqnarray*}
  By~\eqref{eq:Adev} and by induction on $r$, we have 
  $u^r = A(u)^r +O(u^{r+1})$ for all $r\geq 0$. 
Using this identity with $r=\ell$, and rewriting $J(u)=J(0)+O(u)$,
we obtain $u^\ell   J(u) = J(0) A(u)^\ell +O(u^{\ell+1})$,
  so  that:
  \begin{eqnarray*}
    H(u) = \sum_{j=0}^{\ell} g_j(z) A(u)^j + O(u^{\ell+1}),
  \end{eqnarray*}
  with $g_\ell(z):=J(0)$. Thus~\eqref{eq:implicit-functions} holds for $\ell+1$.

\smallskip
  Let us now prove~\eqref{eq:Lambdaiter-A}. Again, we proceed by
  induction on $k\ge 0$. The identity  clearly holds for $k=0$. Assume
  it holds for some $k\geq 0$. 
In~\eqref{eq:Lambdaiter-A}, replace
  $H(u)$ by its expression~\eqref{eq:implicit-functions} obtained with
  $\ell=k+1$, and let $u$ tend to $0$: this  shows that
  $g_k(z)$ is in fact $\Lambda^{(k)}H(0)$. From the definition of
  $\Lambda$ one then   obtains 
$$
\Lambda^{(k+1)}H
  (u)=\frac{\Lambda^{(k)}H(u)-g_k(z)}{A(u)}=\frac{1}{A(u)^{k+1}}\left(H(u)
    - \sum_{j=0}^{k} g_j(z)A(u)^{j}\right).
$$
  Thus~\eqref{eq:Lambdaiter-A} holds for $k+1$.
\end{proof}

\begin{proof}[Proof of Proposition~{\rm\ref{prop:combi-lin-m}}]
Thanks to Lemma~\ref{lemma:implicit-functions} (applied with 
$\GK=\mathbb{Q}(y)$), we can
rewrite~\eqref{eq:Gtilde} as
\begin{eqnarray}\label{eq:Gy-poly}
\frac{\partial \tilde G}{\partial y} (u,y)=
z v \tilde G(u,y) - zv \sum_{j=0}^{m-1} g_j(y)A(u)^j,
\end{eqnarray}
with $v=(1+u)^{m+1}\bu^m$, and for all $j\geq 0$ the series $g_j(y)\equiv
g_j(z;y)$
belongs to $\GK(y)[ [ z ]
] $ (one can actually show that $g_j(y)\in \GK[y][ [ z ] ] $ but we will not
need this). 
As was done in Section~\ref{sec:m11}, we are going to use the fact that
$v(u_i)=v$ for all  $i\in\llbracket 0, m\rrbracket$ to eliminate the  $m$
unknown series $g_j(y)$. For $0\le i \le m$,  the
substitution $u\mapsto u_i$ in~\eqref{eq:Gy-poly} gives:
\begin{eqnarray}\label{eq:Gy-poly-ui}
\frac{\partial \tilde G}{\partial y} (u_i,y)=
z v \tilde G(u_i,y) - zv \sum_{j=0}^{m-1} g_j(y)A_i^j=
z v \tilde G(u_i,y) - zv Q(A_i)
\end{eqnarray}
where $Q(X)= \sum_{j=0}^{m-1} g_j(y)X^{j}$ 
is a polynomial in $X$ of degree less than $m$. 
 Consider the linear combination
\beq\label{combine}
L(u,y):=
\sum_{i=0}^m \frac{\tilde G(u_i,y)}{\prod_{j\neq i}{(A_i-A_j)}}.
\eeq
Then   by~\eqref{eq:Gy-poly-ui},
$$
\begin{array}{llll}
\displaystyle \frac{\partial L}{\partial y} (u,y)&=&
\displaystyle zv L(u,y) - 
zv \sum_{i=0}^m \frac{Q(A_i)}{\prod_{j\neq i}{(A_i-A_j)}} &
\\
&=&
\displaystyle zv L(u,y) &  \hskip 10mm \hbox{ by~\eqref{eq:Lagrange-poly}.}
\end{array}
$$
This homogeneous linear differential equation is readily solved:
$$
L(u,y)=L(u,0) e^{zvy}.
$$
Recall the expression~\eqref{combine} of $L$ in terms of $\tG$.
The initial 
condition~\eqref{init:Gtilde} can be rewritten $\tG(u,0)=v A(u)^m$, which yields
$$
\begin{array}{llll}
L(u,0)& =& \displaystyle v\sum_{i=0}^m \frac{ {A_i}^m}{\prod_{j\neq i}{(A_i-A_j)}} 
\\
&=& v & 
\end{array}
$$
by~\eqref{eq:Lagrange-inv}. Hence $L(u,y)=ve^{zyv}$, and the proposition is proved.
\end{proof}

\subsection{Reconstruction of $\tG(u,y)$}
We are now going to prove that~\eqref{eq:combi-lin-y}, together with the
condition $\tG(-1,y)=0$ derived from~\eqref{eq:Gtilde}, characterizes the
series $\tG(u,y)$. We will 
actually obtain 
a (complicated) expression for
this series, generalizing~\eqref{Gsol-1}. 

We first introduce some notation. Consider a \fps\ in $z$, denoted $H(z;u)\equiv
H(u)$, having coefficients in $\GK[u]$
  for some field  $\GK$ of characteristic $0$ (for instance
  $\qs(y)$).
We define a series $H_k$ in $z$ whose coefficients are symmetric
functions  of $k+1$ variables $x_0, \ldots, x_k$: 
$$
H_k(x_0, \ldots, x_k)= \sum_{i=0}^k \frac{H(x_i)}{\displaystyle \prod_{0\le j \le k,
    j\not = i} (A(x_i)-A(x_j))},
$$
where, as above, $A$ is defined by~\eqref{A-def}.
\begin{Lemma}\label{lem:Hk}
  The series $H_k(x_0, \ldots, x_k)$ has coefficients in $\GK[x_0,
    \ldots , x_k]$. If, moreover, $H(-1)=0$,
then
the coefficients of $H_k$ are multiples of $(1+x_0) \cdots (1+x_k)$.
\end{Lemma}
\begin{proof}
  Observe that
\beq\label{AB}
\frac 1 {A(x_i)-A(x_j)}= \frac 1{x_i-x_j} B(x_i,x_j),
\eeq
where $B(x_i,x_j)$ is a series in $z$ with \emm polynomial,
coefficients in $x_i$ and $x_j$.
Hence 
$$
 H_k(x_0, \ldots, x_k) \prod_{0\le i<j\le k} (x_i-x_j)
$$
has  polynomial coefficients in the $x_i$'s. But these polynomials are
anti-symmetric in the $x_i$'s (since $H_k$ is symmetric), hence they must be
multiples of the Vandermonde $\prod_{i<j} (x_i-x_j)$. Hence $H_k(x_0,
\ldots, x_k)$ has polynomial coefficients.

\medskip
A stronger property than~\eqref{AB} actually holds, namely:
$$
\frac 1 {A(x_i)-A(x_j)}= \frac {(1+x_i)(1+x_j)}{x_i-x_j}\, C(x_i,x_j),
$$
where $C(x_i,x_j)$ is a series in $z$ with  polynomial
coefficients in $x_i$ and $x_j$. Hence, if $H(x)=(1+x)K(x)$,
$$
H_k(x_0, \ldots, x_k)= \sum_{i=0}^k 
{K(x_i)}(1+x_i)^{k+1}\prod_{j\not = i}\frac{(1+x_j)C(x_i,x_j)}{ x_i-x_j}.
$$
Setting $x_0=-1$ shows that $H_k(-1,x_1, \ldots, x_k)=0$, so that
$H_k(x_0, \ldots, x_k)$ is a multiple of $(1+x_0)$. By symmetry, it is
also a multiple of all $(1+x_i)$, for $1\le i \le k$.
\end{proof}

Our treatment of~\eqref{eq:combi-lin-y} actually applies to equations
with an arbitrary right-hand side. We consider a  \fps\ series  $H(z;u)\equiv
H(u)$ with coefficients in $\GK[u]$,
  satisfying $H(-1)=0$ and 
$$
\sum_{i=0}^m \frac{H(u_i)}{\prod_{j\neq i} (A_i-A_j)} = \Phi_m(v),
$$
for some series $\Phi_m(v)\equiv\Phi_m(z;v)$ with coefficients in
$v\GK[v]$, where $v=(1+u)^{m+1}\bu^m$.
Using the above notation, this equation can be rewritten as
$$
H_m(u_0, \ldots, u_m)= \Phi_m(v) .
$$
We will give an explicit expression of $H(u)$ involving two standard
families of symmetric functions~\cite[Chap.~7]{stanley-vol2}, namely the
homogeneous functions $h_\lambda$ 
and the monomial functions $m_\lambda$. We  denote by
$\ell(\lambda)$ the number of parts in a
partition $\lambda$, and  by $\Sn_{m+1}$ the symmetric group on
$\{0,1, \ldots, m\}$.
\begin{Proposition}\label{prop:extraction}
  Let $H(z;u)\equiv H(u)$ be a power series in $z$ with coefficients in $\GK[u]$,
  satisfying $H(-1)=0$ and
\beq\label{eq-sym}
H_m(u_0, \ldots, u_m)= \Phi_m(v) ,
\eeq
where $\Phi_m(v)\equiv\Phi_m(z;v)$ is a series in $z$ with coefficients in
$v\GK[v]$. 

There exists a unique sequence $\Phi_0, \ldots, \Phi_m$ of series in $z$ with
coefficients in $v\GK[v]$ such that  for $0\le k \le m$, and for all
\p\ $\sigma\in \Sn_{m+1}$, 
\beq\label{eq-sym-k}
H_k(u_{\sigma(0)}, \ldots, u_{\sigma(k)})= \sum_{j=k}^m \Phi_j(v)
h_{j-k}(A_{\sigma(0)}, \ldots, A_{\sigma(k)}).
\eeq
In particular, $H(u)\equiv H_0(u)$ is completely determined:
\beq\label{Hsol}
H(u)= \sum_{j=0}^m \Phi_j(v) A(u)^j.
\eeq
The series   $\Phi_k(v)\equiv \Phi_k(z;v)$  can be computed by a
descending induction on $k$ as follows. Let us denote by
$\Phi_{k-1}^>(u)$ the positive part in $u$ of $\Phi_{k-1}(v)$, that is
$$
\Phi_{k-1}^>(u):=  [u^{> }]\Phi_{k-1}(\bu^m(1+u)^{m+1}).
$$
Then for $1\le k\le m$, this series can be expressed in terms of
$\Phi_k, \ldots, \Phi_m$:
\begin{eqnarray}
\Phi_{k-1}^>(u)&=&
-\frac 1{{m\choose k} }[u^{>}] \left(\sum_{j=k}^m \Phi_j(v) \sum_{\lambda \vdash j-k+1} {m-\ell(\lambda)
  \choose k -\ell(\lambda)} m_\lambda(A_1, \ldots, A_m)\right),\label{phi-rec}
\end{eqnarray}
and $\Phi_{k-1}(v)$ can be expressed in terms of $\Phi_{k-1}^>$: 
\beq\label{Phi-k}
\Phi_{k-1}(v)= \sum_{i=0}^m \left(  \Phi_{k-1}^>(u_i)-  \Phi_{k-1}^>(-1)\right).
\eeq
\end{Proposition}

We first establish three lemmas dealing with the symmetric functions of
the series $u_i$ defined  in Lemma~\ref{lem:ui}.
\begin{Lemma}\label{lem:elem}
  The elementary symmetric functions of $u_0=u, u_1, \ldots, u_m$ are 
$$
e_{j}(u_0,u_1, \ldots, u_m)= 
(-1)^{j}{m+1\choose j} + v{\mathbbm 1}_{j=1}
$$
with $v=u^{-m}(1+u)^m$.

The elementary symmetric functions of $u_1, \ldots, u_m$ are
$$
e_{m-j}(u_1, \ldots, u_m)=(-1)^{m-j-1}\sum_{p=0}^j{m+1 \choose p} u^{p-j-1}.
$$
In particular, they are polynomials in $1/u$, and so is any symmetric
polynomial in $u_1, \ldots, u_m$.

Finally,
$$
\prod_{i=0}^m(1+u_i)=v.
$$
\end{Lemma}
\begin{proof}
The symmetric functions of the roots of a polynomial can be read from
the coefficients of this polynomial. Hence the first result follows
directly from the equation satisfied by the  
$u_i$'s, for $0\le i\le m$, namely  
$$
(1+u_i)^{m+1}=v u_i^m.
$$
 For the second one, we need to find the equation satisfied by $u_1, \ldots,
u_m$, which is 
  $$
0= \frac{(1+u_i)^{m+1}u^m -(1+u)^{m+1}u_i^m}{u_i-u}=
 u^m u_i^m -\sum_{j=0}^{m-1} u_i^j u^{m-j-1} \sum_{p=0}^j{m+1 \choose p} u^{p}.
$$
The second result follows.

The third one is obtained by evaluating at $X=-1$
the identity
$$
\prod_{i=0}^{m}(X-u_i)=(1+X)^{m+1}-v X^m.
$$
\end{proof}
\begin{Lemma}\label{P-reconstruct}
  Denote $v=\bu^m(1+u)^{m+1}$. Let $P$ be a polynomial. Then $P(v)$ is a
  Laurent polynomial in $u$. Let $P^>(u)$ denote its positive part:
$$
P^>(u):=[u^>]P(v).
$$
Then
\beq\label{P-dev}
P(v)= P(0)+\sum_{i=0}^m (P^>(u_i)-P^>(-1)).
\eeq
\end{Lemma}
\begin{proof}
The right-hand side of~\eqref{P-dev} is a symmetric polynomial of
$u_0, \ldots, u_m$, and thus, by the first part of
Lemma~\ref{lem:elem}, a polynomial in $v$. Denote it by $\tilde
P(v)$. The second part of
Lemma~\ref{lem:elem} implies that the positive part of $\tilde
P(v)$ in $u$ is
$P^>(u_0)=P^>(u)$. That is, $P(v)$ and $\tilde P(v)$ have the same 
positive part in $u$. In other words, the polynomial
$Q:=P-\tilde P$ is such that $Q(v)$ is a Laurent polynomial in $u$ of
non-positive degree. But since $v=(1+u)^{m+1} \bu^m$, the degree in
$u$ of $Q(v)$ 
coincides with the degree of $Q$, and so $Q$  must be a
constant. Finally, by setting $u=-1$ in $\tilde P(v)$, we see that
$\tilde P(0)= P(0)$ (because  $u_i=-1$ for all $i$ when
$u=-1$, as follows for instance from Lemma~\ref{lem:elem}). Hence
$Q=0$ and the lemma is proved. 
\end{proof}
\begin{Lemma}\label{lem:sym}
  Let $0\le k\le m$, and let $R(x_0, \ldots, x_k)$ be a rational function
  in $k+1$ variables $x_0, \ldots, x_k$, such that for any \p\ $\sigma
  \in \Sn_{m+1}$,
$$
R(u_0, \ldots, u_k)= R(u_{\sigma(0)}, \ldots, u_{\sigma(k)}).
$$
Then there exists a rational fraction in $v$ equal
to  $R(u_0, \ldots, u_k)$.
\end{Lemma} 
\begin{proof} 
  Let $\tilde R$ be the following rational function in $x_0, \ldots, x_m$:
$$
\tilde R (x_0, \ldots, x_m)= \frac 1 {(m+1)!}
\sum_{\sigma  \in \Sn_{m+1}} R(x_{\sigma(0)},
\ldots, x_{\sigma(k)}).
$$
Then $\tilde R$ is a symmetric function of $x_0, \ldots, x_m$, and
hence a rational function in the elementary symmetric functions
$e_j(x_0, \ldots, x_m)$, say $S(e_1(x_0, \ldots, x_m),  \ldots, e_{m+1}(x_0, \ldots, x_m))$. 
By assumption, 
$$
\tilde R (u_0, \ldots, u_m)=S(e_1(u_0, \ldots, u_m),  \ldots,
e_{m+1}(u_0, \ldots, u_m))= R(u_0, \ldots, u_k). 
$$
  Since $S$ is a rational function, it follows from the first part of
  Lemma~\ref{lem:elem} that $ R(u_0, 
  \ldots, u_k)$ can be written as a rational function in $v$.
\end{proof}

\begin{proof}[Proof of Proposition~{\rm\ref{prop:extraction}}]
  We prove~\eqref{eq-sym-k} by descending induction on $k$. For
  $k=m$,~\eqref{eq-sym-k} holds by assumption when $\sigma$ is the
  identity, and actually for any $\sigma$ as $H_m(x_0, \ldots, x_m)$ is a
  symmetric function of the $x_i$'s. Let us assume~\eqref{eq-sym-k}
  holds for some $k>0$, and
  prove it for $k-1$.

Observe that
$$
(A(x_{k-1})-A(x_k))H_k(x_0, \ldots, x_k)= H_{k-1}(x_0, \ldots, x_{k-2}, x_{k-1})-
H_{k-1}(x_0, \ldots, x_{k-2}, x_{k}).
$$
This is easily proved by collecting the coefficient of $H(x_i)$, for
all $i\in\llbracket 0, k\rrbracket$, in both sides of the equation. We also have, for any
indeterminates $a_0, \ldots, a_m$,
$$
(a_{k-1}-a_k)h_{j-k}(a_0, \ldots, a_k)= h_{j-k+1}(a_0, \ldots,
a_{k-2},a_{k-1})-h_{j-k+1}(a_0, \ldots, a_{k-2},a_{k}).
$$
Hence, multiplying~\eqref{eq-sym-k} by $(A_{\sigma({k-1})}-A_{\sigma(k)})$ gives
\begin{multline*}
  H_{k-1}(u_{\sigma(0)}, \ldots, u_{\sigma(k-2)},u_{\sigma(k-1)})
-\sum_{j=k}^m \Phi_j(v) h_{j-k+1}(A_{\sigma(0)}, \ldots, A_{\sigma(k-2)},A_{\sigma(k-1)})
=\\
H_{k-1}(u_{\sigma(0)}, \ldots, u_{\sigma(k-2)},u_{\sigma(k)})
-\sum_{j=k}^m \Phi_j(v) h_{j-k+1}(A_{\sigma(0)}, \ldots, A_{\sigma(k-2)},A_{\sigma(k)})
,
\end{multline*}
for all $\sigma\in\Sn_{m+1}$.
This implies that the series
$$
H_{k-1}(x_0, \ldots,  x_{k-1})-\sum_{j=k}^m \Phi_j(v)
h_{j-k+1}(A(x_0), \ldots,A(x_{k-1})) 
$$
takes the same value at all points $(u_{\sigma(0)},
\ldots,u_{\sigma(k-1)})$, with $\sigma\in \Sn_{m+1}$. 
Hence, 
by Lemma~\ref{lem:sym}, there exists a series in $z$ with \emm rational,
coefficients in $v$, denoted $\Phi_{k-1}(v)$, such that for all
$\sigma\in \Sn_{m+1}$, 
\beq\label{eq-sym-k-1}
H_{k-1}(u_{\sigma(0)}, \ldots,u_{\sigma(k-1)})
-\sum_{j=k}^m \Phi_j(v) h_{j-k+1}(A_{\sigma(0)}, \ldots,
A_{\sigma(k-1)})= \Phi_{k-1}(v).
\eeq
This is exactly~\eqref{eq-sym-k} with $k$ replaced by $k-1$.

\medskip
The next point we will prove is that the coefficients of
$\Phi_{k-1}$ belong to $v \GK[v]$. In order to do so, we
symmetrize~\eqref{eq-sym-k-1} over $u_0, \ldots, u_m$.
 For any subset
$V=\{i_1, \ldots, i_k\}$ of $\{0, \ldots, m\}$, of cardinality $k$, we
denote $u_V=( u_{i_1}, \ldots, u_{i_k})$ and similarly $A_V=( A_{i_1},
\ldots, A_{i_k})$. By~\eqref{eq-sym-k-1},
\beq\label{e-neg:a}
{m+1\choose k} \Phi_{k-1}(v)
=\sum_{V\subset \{0, \ldots, m\}, |V|=k} H_{k-1}(u_V)-\sum_{j=k}^m
\left(\Phi_j(v) \sum_{V\subset \{0, \ldots, m\}, |V|=k}h_{j-k+1}(A_V) \right).
\eeq
We will prove that both sums in the right-hand side of this equation
are series in $z$ with coefficients in $v\GK[v]$.

Denote $x_V=( x_{i_1}, \ldots, x_{i_k})$. Observe that
$$
\sum_{V\subset \{0, \ldots, m\}, |V|=k} H_{k-1}(x_V)
$$
is a series in $z$ with polynomial
coefficients in $x_0, \ldots, x_m$, which is symmetric in these
variables (Lemma~\ref{lem:Hk}). By Lemma~\ref{lem:elem}, the first sum
in~\eqref{e-neg:a} is thus 
a series in $z$ with \emm polynomial, coefficients in $v$. We still
need to prove that this series even vanishes at $v=0$, that is, at
$u=-1$. But this follows from Lemma~\ref{lem:Hk}, since $u_i=-1$ for all $i$
when $u=-1$.

 Let us now consider the second sum in~\eqref{e-neg:a}, and more
specifically the term
\beq\label{second-sum:a}
\Phi_j(v)  \sum_{V\subset \{0, \ldots, m\}, |V|=k}h_{j-k+1}(A_V) .
\eeq
Recall that $$
A_i= \frac{u_i}{1+u_i}\, {e^{-zu_i}}.
$$
But by Lemma~\ref{lem:elem},
$$
\frac 1 {1+u_i}= \frac 1 v \prod_{0\le j \not = i \le m} (1+u_j).
$$
Hence~\eqref{second-sum:a} can be written as a series in $z$ with
coefficients in $\GK[1/v, u_0, \ldots, u_m]$, symmetric in
$u_0, \ldots, u_m$. By the first part of Lemma~\ref{lem:elem}, these
coefficients belong to 
$\GK[v, 1/v]$. We want to prove that they actually
belong to $v\GK[v]$, that is, that they are not singular at
$v=0$ (equivalently, at $u=-1$) and even vanish at this point.
>From the equation $(1+u_i)^{m+1}= vu_i^m$, it follows that
we can label $u_1, \ldots, u_m$ in such a way
$$
1+u_i= \xi^i(1+u)+o(1+u),
$$
where $\xi$ is a primitive $(m+1)^{\hbox{st}}$ root of unity. Since
$\Phi_j(v)$ is a multiple of $v=\bu^m(1+u)^{m+1}$, and the
symmetric function $h_{j-k+1}$ has degree $j-k+1 \le m$, it follows
that the series~\eqref{second-sum:a} is not singular at $u=-1$, and even
vanishes at this point. Hence its coefficients  belong to $v\GK[v]$.

\medskip
We finally want to obtain an explicit expression  of
$\Phi_{k-1}(v)$. Lemma~\ref{P-reconstruct}, together with $\Phi_{k-1}(0)=0$,
establishes~\eqref{Phi-k}. To express $\Phi^>_{k-1}(u)$, we now
symmetrize~\eqref{eq-sym-k-1} over $u_1, \ldots, u_m$. With the
above notation, 
\beq\label{e-neg}
{m\choose k} \Phi_{k-1}(v)=
\sum_{V\subset \{1, \ldots, m\}, |V|=k} H_{k-1}(u_V)-\sum_{j=k}^m
\left(\Phi_j(v) \sum_{V\subset \{1, \ldots, m\}, |V|=k}h_{j-k+1}(A_V)
\right).
\eeq
As above,
$$
\sum_{V\subset \{1, \ldots, m\}, |V|=k} H_{k-1}(x_V)
$$
is a series in $z$ with polynomial
coefficients in $x_1, \ldots, x_m$, which is symmetric in these
variables. By the second part of Lemma~\ref{lem:elem}, the first sum in~\eqref{e-neg} is thus
a series in $z$ with polynomial coefficients in $1/u$. Since
$\Phi_{k-1}(v)$ has coefficients in $\GK[v]$, and hence in
$\GK[u,1/u]$, the second sum in~\eqref{e-neg} is also a zeries in
$z$ with coefficients in $\GK[u,1/u]$. We can now extract from~\eqref{e-neg} 
 the positive part in $u$, and this gives
$$
{m\choose k} \Phi^>_{k-1}(u)= -[u^{>}]\left(\sum_{j=k}^m
\left(\Phi_j(v) \sum_{V\subset \{1, \ldots, m\}, |V|=k}h_{j-k+1}(A_V)
\right)\right).
$$
One easily checks that, for indeterminates $a_1, \ldots, a_m$, 
$$
\sum_{V\subset \{1, \ldots, m\}, |V|=k}h_{j-k+1}(a_V)=
 \sum_{\lambda \vdash j-k+1} {m-\ell(\lambda)
  \choose k -\ell(\lambda)} m_\lambda(a_1, \ldots, a_m),
$$
so that the above expression of $\Phi^>_{k-1}(u)$ coincides
with~\eqref{phi-rec}. 
\end{proof}

\subsection{The case $y=1$}
As explained in Section~\ref{sec:principle},
Theorem~\ref{thm:main} will be proved if we establish
$\tG(u,1)=G_1(u)$, where  
$$
G_1(u)= (1+u)e^{(m+1)z-(m-1)zu}\left(1+\frac{1-e^{mzu}}u\right).
$$
A natural attempt would be to set $y=1$ in the expression of
$\tG(u,y)$ that can be derived from Proposition~\ref{prop:extraction},
as we did when $m=1$ in Section~\ref{sec:m=y=1}. However, we have not been able
to do so, and will proceed differently.

We have proved in Proposition~\ref{prop:combi-lin-m} that the series
$\tG(u,y)$ 
satisfies~\eqref{eq-sym} with $\Phi_m(v)=v e^{zvy}$. In particular, $\tG(u,1)$
satisfies~\eqref{eq-sym} with $\Phi_m(v)=v e^{zv}$. By
Proposition~\ref{prop:extraction}, this 
equation, together with the initial condition $\tG(-1,1)=0$,
characterizes $\tG(u,1)$. 
It is clear that  $ G_1(-1)=0$. Hence it suffices  to prove is the
following proposition. 

\begin{Proposition}
  The series $G_1(u)$ satisfies~\eqref{eq-sym}  with $\Phi_m(v)=v e^{zv}$. 
\end{Proposition}
\begin{proof}
Note that $G_1(u) = e^{(m+1)z} \left(vA(u)^{m-1}-\frac{1}{A(u)}
\right)$. Using  Lemma~\ref{lemma:Lagrange} with $x_i=A_i$, it follows that
\begin{eqnarray*}
\sum_{i=0}^m \frac{G_1(u_i)}{\prod_{j\neq i} (A_i-A_j)} &=& 
0 + (-1)^{m+1} e^{(m+1)z} \prod_{i=0}^m \frac{1}{A_i}
\quad\quad\mbox{ (by }\eqref{eq:Lagrange-inv-bis}
\mbox{ and } \eqref{eq:Lagrange-poly}\mbox{)}\nonumber \\
&=&
(-1)^{m+1} e^{(m+1)z+z\sum_i u_i} \prod_{i=0}^m \frac{(1+u_i)}{u_i}
\\
&=&ve^{zv}
\end{eqnarray*}
by Lemma~ \ref{lem:elem}.
\end{proof}

\subsection{The trivariate series}
 We have now proved that $\tG(u,y)=G(u,y)$, so that
$F(x,y)=\tG(u,y)$ after the change of variables~\eqref{t-x-param}. As
 shown  in Proposition~\ref{prop:combi-lin-m}, the series
$\tG(u,y)$ 
satisfies~\eqref{eq-sym} with $\Phi_m(v)=v e^{zvy}$. Hence
Proposition~\ref{prop:extraction} gives an explicit, although complicated,
expression of the trivariate series $F(t,x,y)$.

\begin{Theorem}\label{thm:trivariate}
Let  $F^{(m)}(t;x,y)\equiv F(t;x,y)$ be the exponential generating
function of labelled $m$-Tamari intervals, defined by~\eqref{F-def}. Let $z$ and
$u$ be two indeterminates, and write 
$$
t=z e^{-m(m+1)z}
\quad \hbox{and } \quad x=({1+u})e^{-mzu}.
$$
Then $F(t;x,y)$ becomes a series in $z$ with polynomial coefficients
in $u$ and $y$, and this series can be computed by an iterative
extraction of positive parts. More precisely,
$$
F(t;x,y)= \sum_{k=0}^m \Phi_k(v) A(u)^k,
$$
where $v=u^{-m}(1+u)^{m+1}$, $A(u)$ is defined by~\eqref{A-def}, and 
   $\Phi_k(v)\equiv \Phi_k(z;v)$ is a series in $z$ with polynomial
coefficients in $v$. This series can be computed by a
descending induction on $k$ as follows. First,
$\Phi_m(v)=ve^{zyv}$. Then for $k\le m$,
$$
\Phi_{k-1}(v)= \sum_{i=0}^m \left(  \Phi_{k-1}^>(u_i)-  \Phi_{k-1}^>(-1)\right)
$$
where 
\begin{eqnarray*}
\Phi_{k-1}^>(u)&=&[u^>]\Phi_{k-1}(v)\\
&=&
-\frac 1{{m\choose k} }[u^{>}] \left(\sum_{j=k}^m \Phi_j(v) \sum_{\lambda \vdash j-k+1} {m-\ell(\lambda)
  \choose k -\ell(\lambda)} m_\lambda(A(u_1), \ldots, A(u_m))\right),
\end{eqnarray*}
and $u_0=u, u_1, \ldots, u_m$ are the $m+1$ roots of the equation
$(1+u_i)^{m+1}=u_i^m v$. 
\end{Theorem}

\noindent{\bf Remark and examples}\\
The case $k=1$ of the above identity gives
\beq\label{phi0}
[u^>] \Phi_{0}(v)= -\frac 1{m }[u^{>}] \left(\sum_{j=1}^m \Phi_j(v)
\sum_{i=1}^m A(u_i)^j\right).
\eeq
Recall that  $F(t;x,y)=\tG(z;u,y)$ has polynomial coefficients in $u$ and $y$. Hence
\begin{eqnarray*}
  F(t;x,y)&=&F(t;1,y)+ [u^>] \left( \sum_{k=0}^m \Phi_k(v)A(u)^k\right) \hskip 10mm
(\mbox{since } u=0 \mbox{ when } x=1)
\\
&=& F(t;1,y)+ [u^>] \left( \sum_{k=1}^m \Phi_k(v)\left( A(u)^k -\frac 1{m }
\sum_{i=1}^m A(u_i)^k\right)\right) \hskip 10mm
(\mbox{by } \eqref{phi0})
\\
&=&
({1+u}) [u^{\ge }]\left(\sum_{k=1}^m \frac{\Phi_k(v)}{1+u} \left(
  A(u)^k-\frac 1{m }\sum_{i=1}^m A(u_i)^k\right)\right)
\end{eqnarray*}
by~\eqref{1+u}, and given that $F(t;x,y)=0$ when $u=-1$.

\smallskip
\noindent $\bullet$
When $m=1$, this is the expression~\eqref{F-param-y1} of
$F^{(1)}(t;x,y)$ (recall that  $\Phi_m= ve^{zyv}$).

\noindent $\bullet$
When $m=2$, the \gf\ of labelled $2$-Tamari intervals satisfies
$$
 \frac{F^{(2)}(t;x,y)}{1+u}=[u^{\ge }]\left(\frac{\Phi_1(v)}{1+u} \left(
  A(u)-\frac {A(u_1)}{2 }-\frac {A(u_2)}{2 }\right)
+(1+\bu)^2e^{zyv} \left(
  A(u)^2-\frac {A(u_1)^2}{2 }-\frac {A(u_2)^2}{2 }
\right)\right),
$$
where 
$$
A(u)=\frac{u}{1+u}e^{-zu},\quad \quad 
u_{1,2}=\frac{1+3u\pm (1+u)\sqrt{1+4u}}{2u^2},
$$
and
$$
\Phi_1(v)=   \Phi_{1}^>(u)+\Phi_{1}^>(u_1)+\Phi_{1}^>(u_2)- 3\Phi_{1}^>(-1),
$$
with
$$
\Phi_1^>(u)=-[u^>]\left( (1+u)^3\bu^2e^{zy(1+u)^3\bu^2}\left(A(u_1)+A(u_2)\right)\right).
$$
This expression has been checked with {\sc Maple}, after computing
the first coefficients of $F(t;x,y)$ from the functional equation~\eqref{eq:Fb}.

\section{Final comments}
\label{sec:final}
\noindent{\bf A constructive proof?} Our proof would not
have been possible without a preliminary task consisting in  \emm guessing,  the
expression of $F(t;x,1)$.  This turned out to be difficult, in particular
because the standard tools like the {\sc Maple} package {\sc Gfun} can only guess D-finite
\gfs, while the \gf\ of the numbers~\eqref{number} is not
D-finite. 
More precisely, the expression of $F(t;x,1)$ \emm becomes
D-finite, after the change of variables~\eqref{t-x-param}, but what is hard 
to guess is this change of variables.
A constructive proof of our result would be most welcome.

\medskip
\noindent{\bf A $q$-analogue of the functional equation.} As described in the introduction, the numbers~\eqref{number} are
conjectured to give the dimension  of certain polynomial rings
generalizing $\DR_{3  ,n}$. These rings are tri-graded (with respect to
the sets of variables 
$\{x_i\}$, $\{y_i\}$ and $\{z_i\}$), 
 and it is conjectured~\cite{bergeron-preville}
that the dimension of the homogeneous component in the $x_i$'s of degree $k$ 
is the number of labelled intervals $[P,Q]$ in $\cT_{n}^{(m)}$ such that the
longest chain from $P$ to $Q$, in the Tamari order, has length
$k$. One can recycle the recursive description of intervals described
in Section~\ref{sec:eq} to generalize the functional equation of
Proposition~\ref{prop:eq}, taking into account (with a new variable $q$) this
distance. Eq.~\eqref{eq:Fb} remains valid, upon defining the operator
$\Delta$ by
$$
\Delta S(x)=\frac{S(qx)-S(1)}{qx-1}.
$$
The coefficient of $t^n$ in the series $F(t,q;x,y)$ does not seem to
factor, even when $x=y=1$. The coefficients of the bivariate series
$F(t,q;1,1)$ have large prime factors.

\medskip

\noindent{\bf Further developments.} 
There is a natural action of the symmetric group $\Sn_n$ on labelled $m$-Tamari
intervals of size $n$: it consists in permuting the labels according
to the permutation one considers, and then to rearrange the labels in
each sequence of consecutive steps so that  they increase. The
dimension of this representation of $\Sn_n$  is the
number~\eqref{number} of labelled  $m$-Tamari intervals of size
$n$. Bergeron and Préville Ratelle have a refined conjecture that
gives the \emm character, of this representation\cite{bergeron-preville}.
 We have very recently proved this conjecture~\cite{mbm-chapuy-preville-prep}. 

\spacebreak

\bigskip
\noindent{\bf Acknowledgements.} We are grateful to François Bergeron
for advertising in his  lectures the conjectural interpretation of the
numbers~\eqref{number} in terms of labelled Tamari intervals.
   We also thank \'Eric Fusy and Gilles Schaeffer   for
interesting discussions on this topic, and  thank \'Eric once more for
allowing us to reproduce some figures
of~\cite{bousquet-fusy-preville}. 
%

\bibliographystyle{plain}
\bibliography{tamar.bib}

\spacebreak

\end{document}

%% file: concatenation.pstex_t
\begin{picture}(0,0)%
\includegraphics{concatenation.pstex}%
\end{picture}%
\setlength{\unitlength}{4144sp}%
\begingroup\makeatletter\ifx\SetFigFont\undefined%
\gdef\SetFigFont#1#2#3#4#5{%
  \reset@font\fontsize{#1}{#2pt}%
  \fontfamily{#3}\fontseries{#4}\fontshape{#5}%
  \selectfont}%
\fi\endgroup%
\begin{picture}(5247,3132)(436,-1750)
\put(451,884){\makebox(0,0)[lb]{\smash{{\SetFigFont{12}{14.4}{\familydefault}{\mddefault}{\updefault}{\color[rgb]{0,0,0}$I_1$}%
}}}}
\put(1711,1199){\makebox(0,0)[lb]{\smash{{\SetFigFont{12}{14.4}{\familydefault}{\mddefault}{\updefault}{\color[rgb]{0,0,0}$Q_1$}%
}}}}
\put(1216, 74){\makebox(0,0)[lb]{\smash{{\SetFigFont{12}{14.4}{\familydefault}{\mddefault}{\updefault}{\color[rgb]{0,0,0}$P_1^\ell$}%
}}}}
\put(2386, 74){\makebox(0,0)[lb]{\smash{{\SetFigFont{12}{14.4}{\familydefault}{\mddefault}{\updefault}{\color[rgb]{0,0,0}$P_1^r$}%
}}}}
\put(5626,884){\makebox(0,0)[lb]{\smash{{\SetFigFont{12}{14.4}{\familydefault}{\mddefault}{\updefault}{\color[rgb]{0,0,0}$I_2$}%
}}}}
\put(4456,1199){\makebox(0,0)[lb]{\smash{{\SetFigFont{12}{14.4}{\familydefault}{\mddefault}{\updefault}{\color[rgb]{0,0,0}$Q_2$}%
}}}}
\put(4411,209){\makebox(0,0)[lb]{\smash{{\SetFigFont{12}{14.4}{\familydefault}{\mddefault}{\updefault}{\color[rgb]{0,0,0}$P_2$}%
}}}}
\put(3196,-1681){\makebox(0,0)[lb]{\smash{{\SetFigFont{12}{14.4}{\familydefault}{\mddefault}{\updefault}{\color[rgb]{0,0,0}$P$}%
}}}}
\put(811,-1141){\makebox(0,0)[lb]{\smash{{\SetFigFont{12}{14.4}{\familydefault}{\mddefault}{\updefault}{\color[rgb]{0,0,0}$I$}%
}}}}
\put(4051,-601){\makebox(0,0)[lb]{\smash{{\SetFigFont{12}{14.4}{\familydefault}{\mddefault}{\updefault}{\color[rgb]{0,0,0}$Q$}%
}}}}
\end{picture}%